\documentclass[11pt,reqno]{amsart}
\usepackage{cmap}
\usepackage[T1]{fontenc}
\usepackage[utf8]{inputenc}
\usepackage{amsmath,amssymb,latexsym,soul,cite,amsfonts,amsthm}
\usepackage{xcolor,enumitem,graphicx}

\usepackage{caption} 
\allowdisplaybreaks
\usepackage{mathtools}

\usepackage{enumitem}

\definecolor{CherryRed}{rgb}{0.7,0.1,0.1}
\definecolor{EgyptBlue}{rgb}{0.1,0.1,0.6}

\usepackage{url}
\usepackage[colorlinks=true,urlcolor=olive,
citecolor=CherryRed,linkcolor=EgyptBlue,linktocpage,pdfpagelabels,
bookmarksnumbered,bookmarksopen]{hyperref}
\usepackage[english]{babel}
\usepackage[T1]{fontenc}
\usepackage[hyperpageref]{backref}

\usepackage{varwidth}
\usepackage{a4wide}
\numberwithin{equation}{section}

\newtheorem{thm}{Theorem}[section]

\theoremstyle{plain}
\newtheorem{lem}[thm]{Lemma}
\theoremstyle{plain}
\newtheorem{prop}[thm]{Proposition}
\theoremstyle{plain}

\theoremstyle{plain}
\theoremstyle{definition}

\newtheorem{rem}[thm]{Remark}

\newcommand{\R}{{\mathbb R}}

\DeclareMathOperator*{\esssup}{ess\,sup}

\title{On the Cheeger problem for rotationally invariant domains}

\author[V. Bobkov]{Vladimir Bobkov$^*$}
\author[E. Parini]{Enea Parini}

\address[E. Parini]{Aix-Marseille Univ, CNRS, Centrale Marseille, I2M, 39 Rue Frederic Joliot Curie, 13453 Marseille, France}\email{enea.parini@univ-amu.fr}
\address[V. Bobkov]{Department of Mathematics and NTIS, Faculty of Applied Sciences, University of West Bohemia, Univerzitn\'i 8, 301 00 Plze\v{n}, Czech Republic\\
Institute of Mathematics, Ufa Federal Research Centre, RAS, Chernyshevsky str. 112, 450008 Ufa, Russia} \email{bobkov@matem.anrb.ru}

\thanks{$^*$\textit{Corresponding author}: bobkov@matem.anrb.ru}


\begin{document}
\begin{abstract}
	We investigate the properties of the Cheeger sets of rotationally invariant, bounded domains $\Omega \subset \mathbb{R}^n$. For a rotationally invariant Cheeger set $C$, the free boundary $\partial C \cap \Omega$ consists of pieces of Delaunay surfaces, which are rotationally invariant surfaces of constant mean curvature. We show that if $\Omega$ is convex, then the free boundary of $C$ consists only of pieces of spheres and nodoids. This result remains valid for nonconvex domains when the generating curve of $C$ is closed, convex, and of class $\mathcal{C}^{1,1}$. Moreover, we provide numerical evidence of the fact that, for general nonconvex domains, pieces of unduloids or cylinders can also appear in the free boundary of $C$.
\end{abstract}

\maketitle

\section{Introduction}

Let $\Omega$ be a bounded domain in $\mathbb{R}^n$, $n \geq 2$. 
The \textit{Cheeger problem} consists in finding subsets $C$ of $\Omega$ which solve the minimization problem
\begin{equation}\label{eq:cheeger}
h(\Omega) = \inf_{E \subset \Omega} \frac{P(E)}{|E|},
\end{equation}
where $P(E)=P(E;\mathbb{R}^n)$ is the distributional perimeter of a subset $E$ of $\Omega$ measured with respect to $\mathbb{R}^n$ and $|E|$ stands for the Lebesgue measure of $E$. The value of $h(\Omega)$ is called \textit{Cheeger constant} of $\Omega$, and any minimizer $C$ of \eqref{eq:cheeger} is called \textit{Cheeger set} of $\Omega$. 
An overview of general properties of the Cheeger problem, such as the existence of the Cheeger set and its regularity, can be found in surveys \cite{leonardi,Parini2011}, see also Section \ref{sec:cheeger} below.
Let us particularly note that the Cheeger set always exists and if $\partial C \cap \Omega$ is nonempty, then it is a constant mean curvature surface (CMC surface) with mean curvature
\begin{equation}\label{eq:H}
H = \frac{h(\Omega)}{n-1}.
\end{equation}
Hereinafter, $\partial C \cap \Omega$ will be called \textit{free boundary} of $C$. 

Despite the geometric nature of the problem \eqref{eq:cheeger}, an explicit analytical description of Cheeger sets is, in general, a difficult task. 
Such description is relatively well-established in the planar case, thanks to the fact that the only planar CMC surface is a circular arc, see \cite{kawohllachand,kreicirik,leonardisaracco} and references therein. 
In particular, if $\Omega \subset \mathbb{R}^2$ is convex, then its Cheeger set $C$ is unique and can be characterized by ``rolling'' a disk $B_r(x)$ inside $\Omega$: 
\begin{equation}\label{eq:cunion}
C = \bigcup\limits_{x \in \Omega^r} B_r(x),
\end{equation}
where $r = \frac{1}{h(\Omega)}$ and $\Omega^r = \{x \in \Omega: \text{dist}(x, \partial \Omega) \geq r\}$, see \cite{kawohllachand}. 
On the other hand, in the higher dimensional case $n \geq 3$, there is a big variety of CMC surfaces, and a characterization of $C$ by ``rolling'' a ball inside $\Omega$ as in \eqref{eq:cunion} can be violated, see \cite[Remark 13]{kawohlfridman}.
Moreover, an explicit characterization of Cheeger sets seems to be known only for some particular domains such as ellipsoids of low eccentricity \cite{ACC}, spherical shells \cite{demengel}, and thin tubular neighbourhoods of smooth closed curves \cite{KLV}, while even for three-dimensional cubes the problem is open \cite{kawohl}. 

The aim of the present work is to make a step towards a better understanding of the Cheeger problem in higher dimensions by investigating the class of domains $\Omega$ which are rotationally invariant with respect to a given vector. 
To the best of our knowledge, the only work explicitly related to this setting is due to Rosales \cite{ros}, where the author studied a general isoperimetric problem under the rotational symmetry of $\Omega$. We will indicate several results from \cite{ros} in more details below.
Assume that some Cheeger set $C$ of such domain $\Omega$ inherits the rotational symmetry and the free boundary $\partial C \cap \Omega$ is nonempty. Then $\partial C \cap \Omega$ consists of pieces of the so-called \textit{Delaunay surfaces}, i.e., rotationally invariant CMC surfaces. 
The class of these surfaces was described by Delaunay in $\mathbb{R}^3$, see \cite{Ken,HY} for a discussion and an $n$-dimensional generalization.  

In this paper, we deal with the problem of determining which types of Delaunay surfaces can constitute the free boundary of Cheeger sets of rotationally invariant domains. 
In principle, the only Delaunay surfaces of positive mean curvature are spheres, nodoids, unduloids, and cylinders. 
In Theorem \ref{thm:main}, we prove that if $\Omega$ is a convex, rotationally invariant domain, then the free boundary of its (unique) Cheeger set consists only of pieces of spheres or nodoids. 
This result is generalized in Proposition \ref{prop:unlul-nonexist1} to the case of nonconvex domains which admit Cheeger sets whose generating curve is closed, convex, and of class $\mathcal{C}^{1,1}$.
Moreover, we provide numerical evidence of the fact that, for general nonconvex domains, pieces of unduloids or cylinders can indeed appear in the free boundary of their Cheeger sets. 
Investigation of the Cheeger problem for several model domains (cylinders, cones, double cones) complement our analysis.

\section{Preliminaries and main results}\label{sec:prelim}
We start by reviewing some basic facts about domains of revolution and Delaunay surfaces.
Let $\gamma:[a,b]\to \mathbb{R} \times \mathbb{R}^+$ with $\gamma(s)=(x(s),y(s))$ be a $\mathcal{C}^{1,1}$-curve parametrized by its arc-length, and let $\sigma: [a,b] \to [-\pi,\pi]$ be the angle between the tangent to $\gamma(s)$ and the positive $x$-direction. This implies that the normal vector to $\gamma$ is given by $(\sin \sigma, -\cos \sigma)$. Since $\gamma$ is of class $\mathcal{C}^{1,1}$, we have that $\sigma$ is a Lipschitz-continuous function, and hence it is almost everywhere differentiable.
Let $M \subset \R^n$ ($n \geq 3$) be the embedded surface obtained by rotating the graph of $\gamma$ around the $x$-axis. More precisely, $M$ is invariant with respect to the actions of $SO(n-1)$ fixing the $x$-axis and is defined as
\begin{equation*}\label{eq:M}
M
= 
\left\{
(x(s), y(s) \mathbb{S}^{n-2}) \in \mathbb{R}^n:~ a\leq s \leq b
\right\},
\end{equation*}
where $\mathbb{S}^{n-2} \subset \mathbb{R}^{n-1}$ is an $(n-2)$-sphere. 
Since $M$ is also of class $\mathcal{C}^{1,1}$, the principal curvatures are defined at $\mathcal{H}^{n-1}$-almost every point of $M$ due to Rademacher's theorem. 
Moreover, the mean curvature $H$ of $M$ is a bounded function which coincides with the distributional mean curvature of $M$ as defined in \cite[Section 17.3]{maggi}. 
We have the explicit formula
$$
H(s) = \frac{1}{n-1}\left(-\sigma'(s) + (n-2)\frac{\cos \sigma(s)}{y(s)}\right)
$$ 
for the mean curvature at the point $(x(s), y(s)\mathbb{S}^{n-2})$, measured with respect to the normal vector $(\sin{\sigma},-\cos{\sigma})$. 
Therefore, $\gamma$ can be characterized as a solution of the system
\begin{equation}\label{eq:system}
\left\{ 
\begin{array}{r c l} x'(s) & = & \cos{\sigma(s)},\\ 
y'(s) & = & \sin{\sigma(s)},\\ 
\sigma'(s) & = & -(n-1)H(s) + (n-2)\frac{\cos{\sigma(s)}}{y(s)}.
\end{array}
\right. 
\end{equation}
We refer to \cite[Section 4]{HMR} and \cite[Section 2]{ros} for the case of $\mathcal{C}^2$-surfaces in $\mathbb{R}^n$, and to \cite[Section 2]{dalphinmasnouhenrottakahashi} for the case of $\mathcal{C}^{1,1}$-surfaces in $\mathbb{R}^3$.

Assume now that $\gamma$ can be locally described as a function $y=y(x)$ for $x \in (\alpha,\beta)$. 
In view of Rademacher's theorem, the mean curvature at $\mathcal{H}^{n-1}$-almost every point $(x,y(x)\mathbb{S}^{n-2})$ of the corresponding portion of $M$ can be expressed as
$$
H(x) = \frac{1}{n-1}\left(-\frac{y''}{(1+y'^2)^{3/2}}  +  \frac{(n-2)}{y(1+y'^2)^{1/2}}\right).
$$ 
Equivalently, $y$ satisfies the equation
\begin{equation}\label{eq:hy14}
\frac{y''}{(1+y'^2)^{3/2}} 
- 
\frac{(n-2)}{y(1+y'^2)^{1/2}} + (n-1) H(x) = 0.
\end{equation}

\subsection{Delaunay surfaces}\label{sec:delaunay}
Throughout this subsection, we assume that the mean curvature $H$ of $M$ is a nonnegative constant. 
In this case, $M$ is called \textit{Delaunay surface}, and the system \eqref{eq:system} possesses the first integral
\begin{equation}\label{eq:firstint}
y^{n-2} \cos \sigma - H y^{n-1} = T,
\end{equation}
where $T$ is a constant. It is easy to see that if $H>0$, then
\[ 
T \leq \frac{1}{(n-1)^{n-1}}\left(\frac{n-2}{H}\right)^{n-2}.
\]

Note that in the particular case of $\mathbb{R}^3$, $(x(s), y(s))$ can be conveniently expressed as
\begin{equation}\label{eq:kenmotsu}
(x(s),y(s)) 
= 
\left(
\int_0^s \frac{1+B \cos (2Ht)}{(1+B^2+2B\cos (2Ht))^{1/2}} \, dt,
\frac{(1+B^2+2B\cos (2Hs))^{1/2}}{2H}
\right)
\end{equation}
provided $H \neq 0$, where $B$ is a constant, see \cite[Section 2]{Ken}.
The advantage of \eqref{eq:kenmotsu} consists in the fact that $x(0)=0$, $y(0)=\frac{1+B}{2H}$, and $\frac{dy}{dx}(0)=0$. 
In the case of the equation \eqref{eq:hy14}, the first integral is
\begin{equation}\label{cyl1:hy15}
\frac{y^{n-2}}{(1+y'^2)^{1/2}} - H y^{n-1} = c
\end{equation} 
for some constant $c$, whenever $y' \neq 0$, cf.\ \cite[Eq.\ (5)]{HY}. 
This equation can be resolved as
\begin{equation*}\label{cyl1:hy151}
y' = \pm \left[\left(\frac{y^{n-2}}{c+Hy^{n-1}}\right)^2-1\right]^{1/2}
\end{equation*}
or, in terms of $x=x(y)$, as
\begin{equation}\label{cyl1:hy2}
x(y) = x_0 \pm \int_{y(x_0)}^y \left[\left(\frac{t^{n-2}}{c+Ht^{n-1}}\right)^2-1\right]^{-1/2} \, dt.
\end{equation}

\medskip
Analysing the system \eqref{eq:system} and taking into account its first integral \eqref{eq:firstint} and the fact that $\gamma$ maps to $\mathbb{R} \times \mathbb{R}^+$, we deduce that any solution $\gamma$ of \eqref{eq:system} generates a portion of one of the following six types of Delaunay surfaces which are depicted on Figures \ref{fig:del1}, \ref{fig:del2}, \ref{fig:del3}, cf.\ \cite[Proposition 4.3]{HMR}:
\begin{enumerate}[label={\rm(\roman*)}]
	\item If $H > 0$ and 
	$$
	T = \frac{1}{(n-1)^{n-1}}\left(\frac{n-2}{H}\right)^{n-2},
	$$
	then $\gamma(s)=\left(s,\frac{n-2}{(n-1)H}\right)$, i.e., $\gamma$ generates a \textit{cylinder}.
	\item If $H > 0$ and 
	$$
	0 < T < \frac{1}{(n-1)^{n-1}}\left(\frac{n-2}{H}\right)^{n-2},
	$$
	then $\gamma$ generates an \textit{unduloid}.
	It holds 
	\begin{equation}\label{eq:unduloid_maximum}
	0 < \min_{s \in \mathbb{R}} y(s) < \frac{n-2}{(n-1)H} < \max_{s \in \mathbb{R}} y(s) < \frac{1}{H}.
	\end{equation}
	\item If $H > 0$ and $T=0$, then $\gamma$ generates a \textit{sphere} of radius $\frac{1}{H}$ centered on the $x$-axis. 
	\item If $H > 0$ and $T<0$, then $\gamma$ generates a \textit{nodoid}.
	It holds 
	\begin{equation*}\label{eq:nodoid_maximum}
	0 < \min_{s \in \mathbb{R}} y(s) < \left(\frac{-T}{H}\right)^{\frac{1}{n-1}} \quad \text{and} \quad \max_{s \in \mathbb{R}} y(s) > \max\left\{\frac{1}{H}, \left(\frac{-T}{H}\right)^{\frac{1}{n-1}}\right\}.
	\end{equation*}
	\item If $H = 0$ and $T \neq 0$, then $\gamma$ generates a \textit{catenoid}.
	\item If $H = 0$ and $T=0$, then $\gamma(s) = (\text{const},s)$, i.e., $\gamma$ generates a \textit{hyperplane} perpendicular to the $x$-axis.
\end{enumerate}

Hereinafter, for convenience of writing, we will identify Delaunay surfaces with their respective generating curves.

 \begin{figure}[ht]
 	\centering
 	\includegraphics[width=0.6\linewidth]{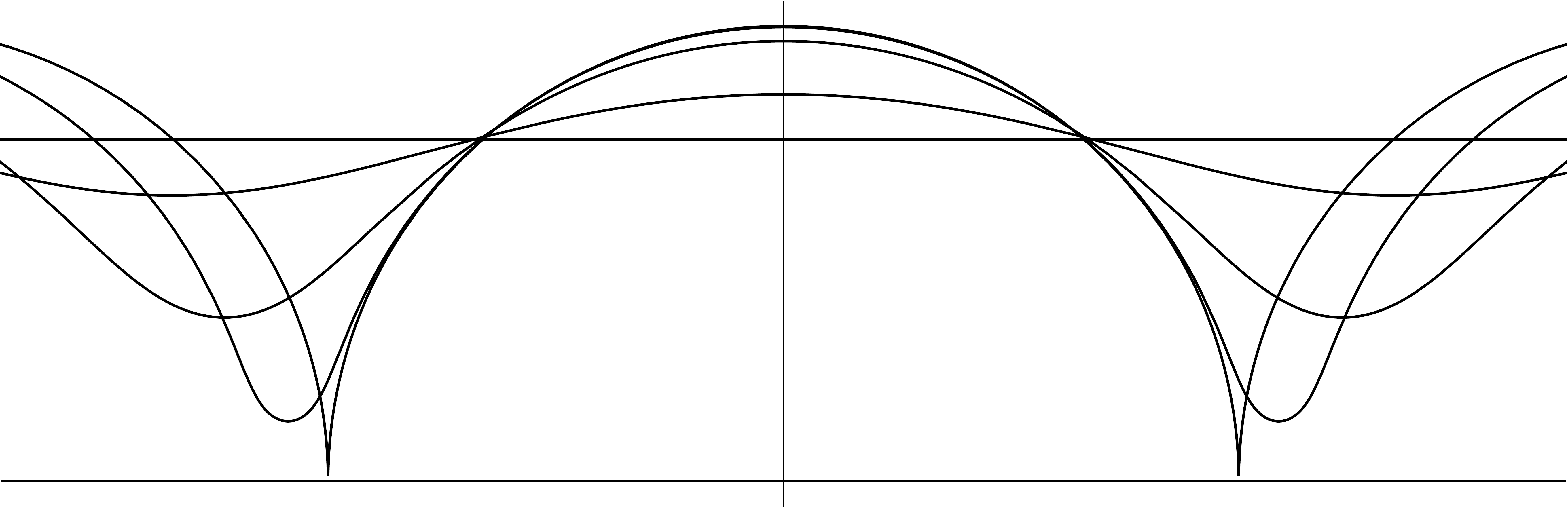}
 	\caption{$n=5$. Generating curves of a cylinder, unduloids, and spheres.}
 	\label{fig:del1}
 \end{figure}

 \begin{figure}[ht]
	\centering
	\includegraphics[width=0.6\linewidth]{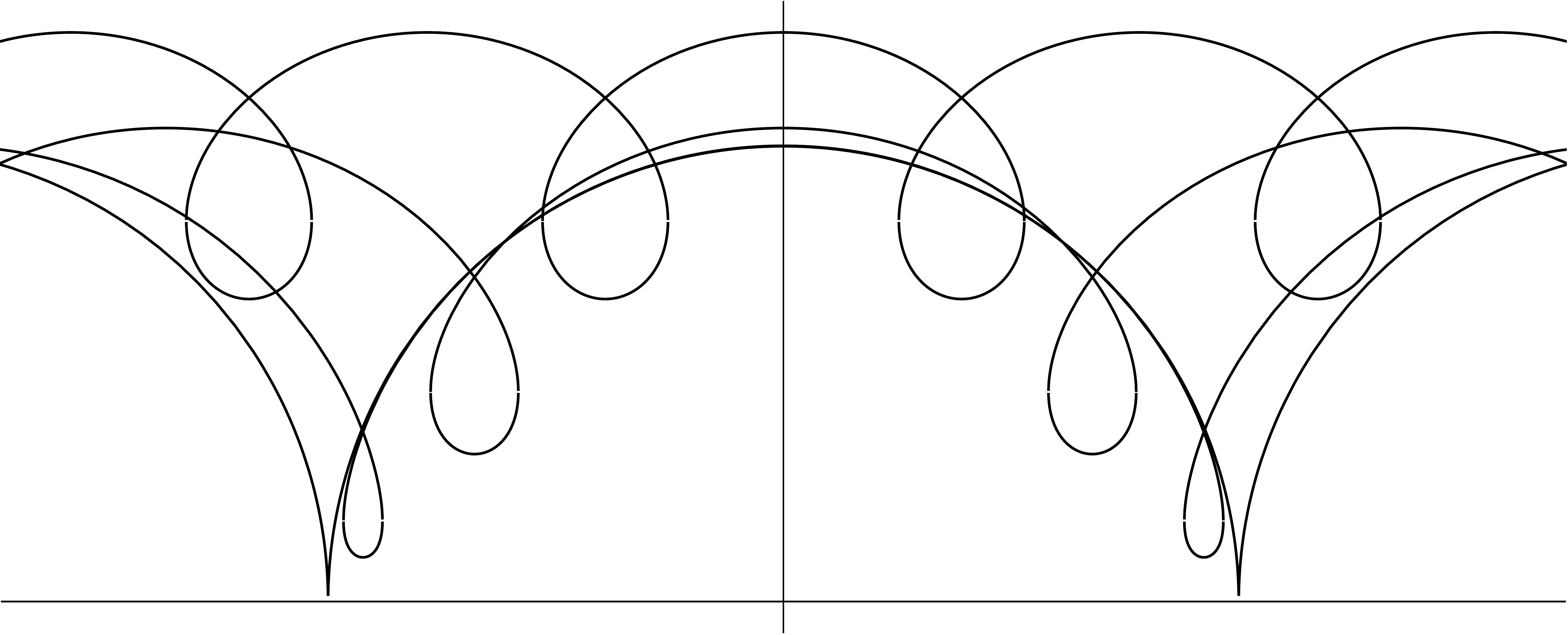}
	\caption{$n=5$. Generating curves of nodoids and spheres.}
	\label{fig:del2}
\end{figure}

 \begin{figure}[ht]
	\centering
	\includegraphics[width=0.6\linewidth]{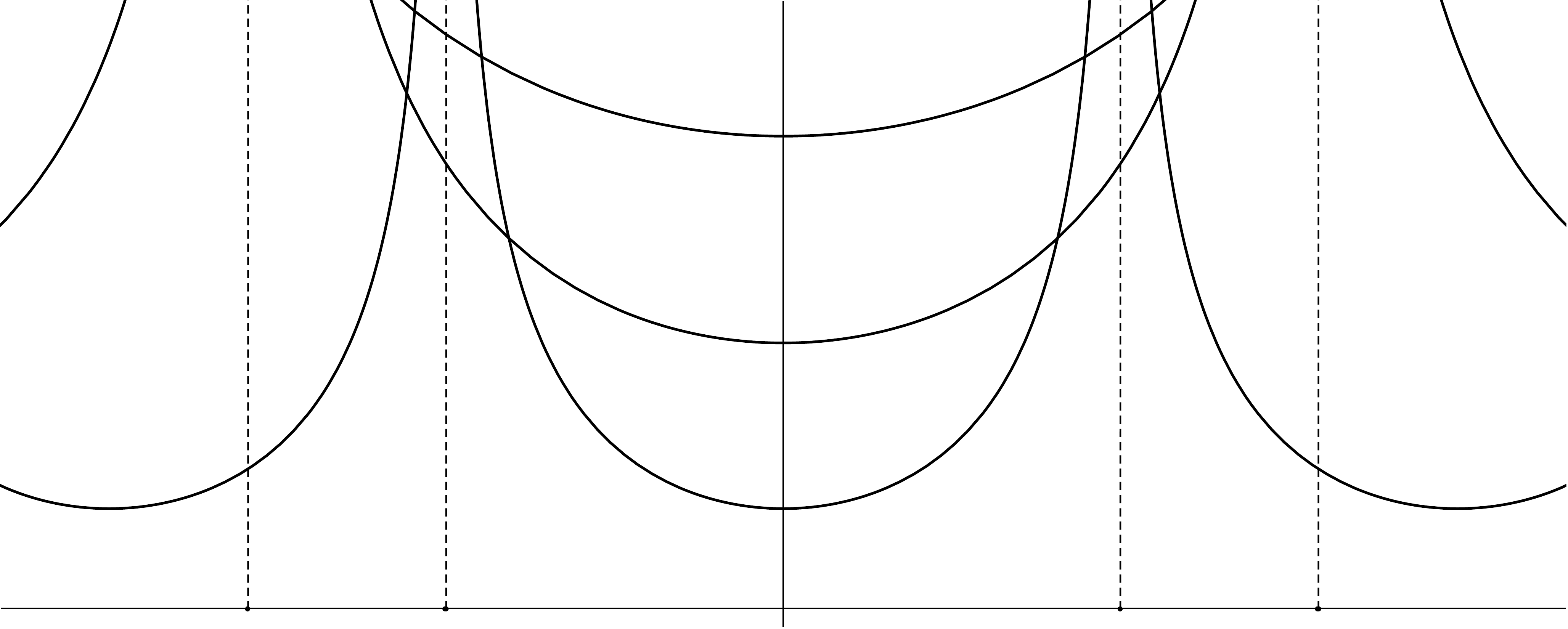}
	\caption{$n=5$. Generating curves of catenoids and vertical hyperplanes (dashed). Simultaneously, the dashed vertical lines are asymptotes of the corresponding catenaries.}
	\label{fig:del3}
\end{figure}

\subsection{Cheeger problem}\label{sec:cheeger}
We start by collecting several general regularity properties of Cheeger sets.
\begin{prop}\label{prop:reg}
Let $\Omega \subset \mathbb{R}^n$ be an open, bounded set. 
If $C$ is a Cheeger set of $\Omega$, then it satisfies the following regularity properties:
\begin{enumerate}[label={\rm(\roman*)}]
 \item\label{prop:reg:1} $\partial C \cap \Omega$ is analytic, except possibly for a singular set of Hausdorff dimension at most $n-8$. At regular points, the mean curvature is constant and equal to $\frac{h(\Omega)}{n-1}$.
 \item\label{prop:reg:2} If $\partial \Omega$ is of class $\mathcal{C}^{1,1}$ in a neighbourhood of a point $x \in \partial \Omega \cap \partial C$, then $\partial C$ is also of class $\mathcal{C}^{1,1}$ around $x$.
\end{enumerate}
\end{prop}
\begin{proof}
Property \ref{prop:reg:1} is proven in \cite[Theorem 2]{gonzalezmassaritamanini}. Property \ref{prop:reg:2} follows by reasoning as in \cite[Theorem 2]{caselleschambollenovaga}, taking into account \cite[Theorem 3.5 (v)]{bobkovparini}.
\end{proof} 

In general, Cheeger sets need not be unique, as can be seen from various examples (see, for instance, \cite[Remark 12]{kawohlfridman} or \cite[Example 4.6]{leonardipratelli}). Nevertheless, if $\Omega$ is convex, then it admits a unique Cheeger set, which is convex, and whose boundary is of class $\mathcal{C}^{1,1}$ \cite{AC}.

\medskip
With the help of information provided above, let us consider the Cheeger problem \eqref{eq:cheeger} in a rotationally invariant, bounded domain $\Omega \subset \mathbb{R}^n$. 
We assume that the boundary $\partial\Omega$ of $\Omega$ can be characterized as
\begin{equation*}\label{eq:Omega}
\partial\Omega 
= 
\left\{
(\xi(s), \eta(s) \mathbb{S}^{n-2}) \in \mathbb{R}^n:~ \alpha \leq s \leq \beta
\right\},
\end{equation*}
where the generating curve $\Gamma : [\alpha,\beta]\to \R \times \R^+$ defined as $\Gamma(s)=(\xi(s),\eta(s))$ is continuous, does not have self-intersections, and is either closed or satisfies $\eta(\alpha)=\eta(\beta)=0$.

The following result can be proved in much the same way as \cite[Lemma 3.1]{ros}\footnote{Although the paper \cite{ros} deals with rotationally invariant, strictly convex domains, the convexity assumption is not used in the proof of \cite[Lemma 3.1]{ros}.}.
\begin{lem}\label{lem:1}
	Let $\Omega \subset \mathbb{R}^n$ be a rotationally invariant, bounded domain, and let $C$ be its Cheeger set. 
	If $C$ is also rotationally invariant, then $\partial C \cap \Omega$ is analytic.
\end{lem}

Notice that the existence of a rotationally invariant Cheeger set required in Lemma \ref{lem:1} can be guaranteed if $\Omega$ is a Schwarz symmetric domain with respect to the $x$-axis, that is, if the intersection of $\Omega$ with any hyperplane $\{x= \text{const}\}$ is either empty or an $(n-1)$-dimensional ball centered on the $x$-axis. Indeed, let $C$ be a Cheeger set of $\Omega$, and let $C^*$ be the Schwarz symmetrization of $C$, as defined in \cite[Section 19.2]{maggi}. 
Then $C^*$ is still a subset of $\Omega$ and 
\[ 
P(C^*) \leq P(C) \quad \text{and} \quad |C^*|=|C|
\]
by \cite[Theorem 19.11]{maggi}, which implies that $C^*$ is also a Cheeger set of $\Omega$.

More can be said if we additionally ask $\Omega$ to be convex.
\begin{lem} \label{propertiesconvex}
	Let $\Omega \subset \mathbb{R}^n$ be a convex, rotationally invariant, bounded domain.	Then $\Omega$ admits a unique Cheeger set, which is convex, rotationally invariant, and with boundary of class $\mathcal{C}^{1,1}$.
\end{lem}
\begin{proof}
    The result is a consequence of the existence of a convex Cheeger set, which was obtained in \cite[Remark 10]{kawohlfridman}, of the uniqueness result proven in \cite[Theorem 1]{AC}, and of the regularity properties obtained in \cite{stredulinskyziemer}.
\end{proof}

Let us now characterize Delaunay surfaces which can form the free boundary of a Cheeger set $C$.
Since the mean curvature of $\partial C \cap \Omega$ is positive, it is clear that neither catenoids nor vertical hyperplanes can appear in the free boundary of $C$. 
In the following theorem, we show that unduloids and the cylinder cannot be parts of $\partial C \cap \Omega$, too, provided $\Omega$ is convex.

\begin{thm}\label{thm:main}
	Let $\Omega \subset \mathbb{R}^n$ be a convex, rotationally invariant, bounded domain. Let $C$ be the Cheeger set of $\Omega$. Then each connected component of $\partial C \cap \Omega$ is a part of a sphere or a nodoid.
\end{thm}
\begin{proof}
By Lemma \ref{propertiesconvex}, $\Omega$ admits a unique Cheeger set $C$, which is convex, rotationally invariant, and with boundary of class $\mathcal{C}^{1,1}$. Let $\gamma: [a,b]\to \R \times \R^+$ defined by $\gamma(s)= (x(s),y(s))$ be a generating curve of $C$. 
Recall that $\gamma$ satisfies the system \eqref{eq:system}. Moreover, for every open ball $B_r$, and for every $F \subset C$ such that $F \bigtriangleup C \subset \subset B_r$, by minimality of the Cheeger set it holds
\[ P(C;B_r) - h(\Omega)|C \cap B_r| \leq P(F;B_r) - h(\Omega)|F \cap B_r|,\]
where, for a set $E \subset \R^n$, $P(E;B_r)$ is the distributional perimeter of $E$ measured with respect to $B_r$. Therefore, by \cite[Proposition 2.1]{bellettinichambollenovaga} the mean curvature satisfies 
\[ \esssup_{s \in [a,b]} H(s) \leq \frac{h(\Omega)}{n-1},\] where the equality holds true for any $s$ such that $(x(s), y(s) \mathbb{S}^{n-2}) \in \partial C \cap \Omega$ (see Proposition \ref{prop:reg}). 
Define the function $T:[a,b]\to \R$ as
\begin{equation*}\label{eq:T}
T(s)= y^{n-2}\cos{\sigma} - \frac{h(\Omega)}{n-1} y^{n-1}.
\end{equation*}
Notice that $T$ is constant on every connected component of $\partial C \cap \Omega$ and coincides there with \eqref{eq:firstint}.
By the regularity of $C$, $T$ is a Lipschitz-continuous function, and therefore it is differentiable almost everywhere on $[a,b]$ with the derivative 
\begin{align}
\notag 
T'(s) & = (n-2)y^{n-3}y'\cos{\sigma} - y^{n-2}\sigma' \sin{\sigma} - h(\Omega) y^{n-2} y'
\\ 
\notag
& = (n-2)y^{n-3}\sin{\sigma}\cos{\sigma} - y^{n-2}\left(-(n-1)H(s) + (n-2)\frac{\cos{\sigma}}{y}\right)\sin{\sigma} 
- h(\Omega)y^{n-2} \sin{\sigma}
\\
\label{eq:Tprime}
&= (n-1)y^{n-2}\left(H(s)-\frac{h(\Omega)}{n-1}\right)\sin{\sigma}.
\end{align}
That is, $T' \leq 0$ if $\sin{\sigma} \geq 0$, and $T' \geq 0$ if $\sin{\sigma} \leq 0$. 
In view of the convexity of $C$, we can assume, without loss of generality, that $y(a)=y(b)=0$, and hence $T(a) = T(b) = 0$.
On the other hand, by the convexity and regularity of $C$, there exists $c \in (a,b)$, such that $\sin{\sigma} \geq 0$ on $[a,c]$, and $\sin{\sigma} \leq 0$ on $[c,b]$.
It then follows from the fundamental theorem of calculus that $T(s) \leq 0$ for all $s \in [a,b]$, which implies that the free boundary $\partial C \cap \Omega$ consists of spheres or nodoids.
\end{proof}

The above proof can be generalized, under suitable assumptions on $C$, to the case of nonconvex, rotationally invariant domains.
\begin{prop}\label{prop:unlul-nonexist1}
	Let $\Omega \subset \mathbb{R}^n$ be a rotationally invariant, bounded domain. 
	Suppose that $\Omega$ admits a rotationally invariant Cheeger set $C$ generated by a closed, convex curve $\gamma :[a,b] \to \R \times \R^+$ of class $\mathcal{C}^{1,1}$. 
	Then each connected component of $\partial C \cap \Omega$ is a part of a nodoid.
\end{prop}
\begin{proof}
Since $\gamma$ is of class $\mathcal{C}^{1,1}$, we can define the function $T$ as in the proof of Theorem \ref{thm:main}.
In view of the fact that $\gamma$ is closed, we can assume, without loss of generality, that $x(a)=x(b)$ and $y(a)=y(b)=\min_{s \in [a,b]} y(s)$. 
Then the convexity and regularity of $\gamma$ yield the existence of $c \in (a,b)$ such that $\sin{\sigma} \geq 0$ on $[a,c]$, and $\sin{\sigma} \leq 0$ on $[c,b]$. Since $\cos{\sigma(a)}=\cos{\sigma(b)}=-1$, we have $T(a)=T(b) \leq 0$, and hence $T(s) \leq 0$ for all $s \in [a,b]$, see \eqref{eq:Tprime}.
If $y(a)>0$, then $T(a)<0$, which implies that $T(s) < 0$ for all $s \in [a,b]$, and the claim of the proposition follows.
Assume that $y(a)=0$. Let $a_1$ and $b_1$ be such that $a \leq a_1 < b_1 \leq b$ and $y(s)=0$ for $s \in [a,a_1] \cup [b_1,b]$ and $y(s)>0$ for $s \in (a_1,b_1)$.
Moreover, let $a_2$ and $b_2$ be such that $a \leq a_1 \leq a_2 \leq b_2 \leq b_1 \leq b$ and $T(s)=0$ for $s \in [a,a_2] \cup [b_2,b]$ and $T(s)<0$ for $s \in (a_2,b_2)$.
If $a_2>a_1$ or $b_2<b_1$, then $\gamma(s)$ for $s \in (a_1,a_2)$ or $s \in (b_2,b_1)$ describes a part of a sphere of radius $\frac{n-1}{h(\Omega)}$ centred on the $x$-axis, which contradicts the regularity of $\gamma$.
Thus, $T(s)<0$ for $s \in (a_1,b_1)$, which completes the proof.
\end{proof}

\begin{rem}
We observe that the above results do not use the fact that $C$ is a Cheeger set, but rather its regularity, a convexity assumption, and the boundedness of the mean curvature. 
If $\Omega$ is smooth, strictly convex, and rotationally invariant, then these properties hold true also for rotationally invariant solutions of the isoperimetric problem
\begin{equation}\label{isoperimetricset} 
	\min \{ P(F):~ F \subset \Omega,\,|F|=V \}
\end{equation}
for fixed $V \in (0,|\Omega|)$. 
Indeed, if $E$ is a rotationally invariant minimizer for \eqref{isoperimetricset}, then $\partial E$ is of class $\mathcal{C}^{1,1}$, convex, and has bounded mean curvature (see \cite[Theorems 1.1, 2.1, and Lemma 3.1]{ros}). 
Therefore, Theorem \ref{thm:main} is an extension of \cite[Lemma 3.4]{ros} about the absence of pieces of unduloids of negative Gauss-Kronecker curvature.
\end{rem}

The claim of Proposition \ref{prop:unlul-nonexist1} can be also obtained provided that the generating curve of $\Omega$ is sufficiently high over the $x$-axis.
Namely, recall from \eqref{eq:unduloid_maximum} that the maximal ordinate of any unduloid of the mean curvature $H$ is less than $\frac{1}{H}$, and a sphere of the same mean curvature has the radius $\frac{1}{H}$.
Therefore, in view of \eqref{eq:H}, the following result takes place.
\begin{prop}\label{lem:unlul-nonexist}
	Let $\Omega \subset \mathbb{R}^n$ be a rotationally invariant, bounded domain, generated by a closed curve $\Gamma :[\alpha,\beta] \to \R \times (0,+\infty)$ given by $\Gamma=(\xi(s),\eta(s))$ such that
	\begin{equation}\label{eq:prop_undul0}
	\min_{\alpha \leq s \leq \beta} \eta(s) \geq \frac{n-1}{h(\Omega)}.
	\end{equation} 
	Suppose that $\Omega$ admits a rotationally invariant Cheeger set $C$. Then each connected component of $\partial C \cap \Omega$ is a part of a nodoid.
\end{prop}

\begin{rem}
	Notice that Proposition \ref{lem:unlul-nonexist} does not require $C$ to be generated by a convex regular curve.
	Moreover, the assumption \eqref{eq:prop_undul0} can be guaranteed if the following, easily verifiable estimate holds true:
	\begin{equation}\label{eq:prop_undul02}
	\min_{\alpha \leq s \leq \beta} \eta(s) \geq \frac{n-1}{n} \left(\frac{|\Omega|}{\omega_n}\right)^\frac{1}{n},
	\end{equation} 
	where $\omega_n$ is the volume of a unit ball in $\mathbb{R}^n$. 
	Indeed, \eqref{eq:prop_undul02} yields \eqref{eq:prop_undul0} due to the Faber-Krahn inequality
	$$
	h(\Omega) \geq h(B) = n \left(\frac{\omega_n}{|\Omega|}\right)^\frac{1}{n},
	$$
	where $B$ is a ball such that $|B|=|\Omega|$, see, e.g., \cite[Corollary 15]{kawohlfridman}.
\end{rem}

\section{Examples}\label{sec:examples}
In this section, we study the Cheeger problem in cylinders, cones, and double cones, and calculate values of the Cheeger constant in several particular cases.

\subsection{Cylinders}\label{sec:cylinder}
Consider the $n$-dimensional cylinder
$$
Z_{l,r} := (0,l) \times B_{r}(0)
\quad \text{for some}~ l, r>0,
$$
where $B_r(0) \subset \mathbb{R}^{n-1}$ is the open ball of radius $r$ centred at the origin. 
We know from Lemma \ref{propertiesconvex} that there exists a unique Cheeger set $C$ of $Z_{l,r}$, which is convex, rotationally invariant, and with boundary of class $\mathcal{C}^{1,1}$. 
The symmetry of $Z_{l,r}$ with respect to the hyperplane $\{x=\frac{l}{2}\}$ and the uniqueness of $C$ imply that $C$ is also symmetric with respect to $\{x=\frac{l}{2}\}$.

The generating curve $\Gamma$ of $Z_{l,r}$ consists of three segments:  $\{0\} \times [0,r]$, $[0,l] \times \{r\}$, and $\{l\} \times [0,r]$. 
Therefore, in view of Theorem \ref{thm:main} and the regularity of $C$, the generating curve $\gamma=(x(s),y(s))$, $s \in [a,b]$, of $C$ smooths both corners of $\Gamma$ either by circular arcs or by parts of a nodoid. In particular, the free boundary $\partial C \cap Z_{l,r}$ has nonempty interior. 
We have 
\begin{equation}\label{eq:minxmaxx}
\min_{a \leq s \leq b} x(s) = 0
\quad \text{and} \quad 
\max_{a \leq s \leq b} x(s) = l,
\end{equation}
since otherwise we could shift the generating curve along the $x$-axis and get a contradiction to the uniqueness of $C$.
Moreover, 
\begin{equation}\label{eq:minxmaxy}
\min_{a \leq s \leq b} y(s) = 0
\quad \text{and} \quad 
\max_{a \leq s \leq b} y(s) = r.
\end{equation}
The first equality trivially follows from the convexity of $C$, and the second equality follows again from the uniqueness of $C$.  
Let us show that only parts of a nodoid can constitute $\partial C \cap Z_{l,r}$. 
Suppose, by contradiction, that the angles of $\Gamma$ are smoothed by circular arcs. In view of \eqref{eq:minxmaxx} and \eqref{eq:minxmaxy}, $\gamma$ consists of the horizontal line $[r,l-r] \times \{r\}$ and two circular arcs of radius $r$ and angle $\frac{\pi}{2}$ joining this horizontal line with the $x$-axis. That is, using \eqref{eq:H}, the fact that the mean curvature of a sphere of radius $r$ equals $\frac{1}{r}$, and explicit formulas for the surface areas and volumes of a ball of radius $r$ and cylinder $Z_{l-2r,r}$, we get
\begin{equation}\label{eq:cyl:sphere}
h(\Omega) = \frac{n-1}{r} = \frac{n \omega_{n} r^{n-1} + (n-1) \omega_{n-1}r^{n-2}(l-2r)}{\omega_{n} r^n + \omega_{n-1} r^{n-1} (l-2r)}.
\end{equation}
Here and below, $\omega_k$ stands for the volume of a unit ball in $\mathbb{R}^k$, $k \geq 2$.
However, it is not hard to see that the second equality in \eqref{eq:cyl:sphere} is impossible.

\begin{figure}[ht]
	\centering
	\includegraphics[height=3.8cm]{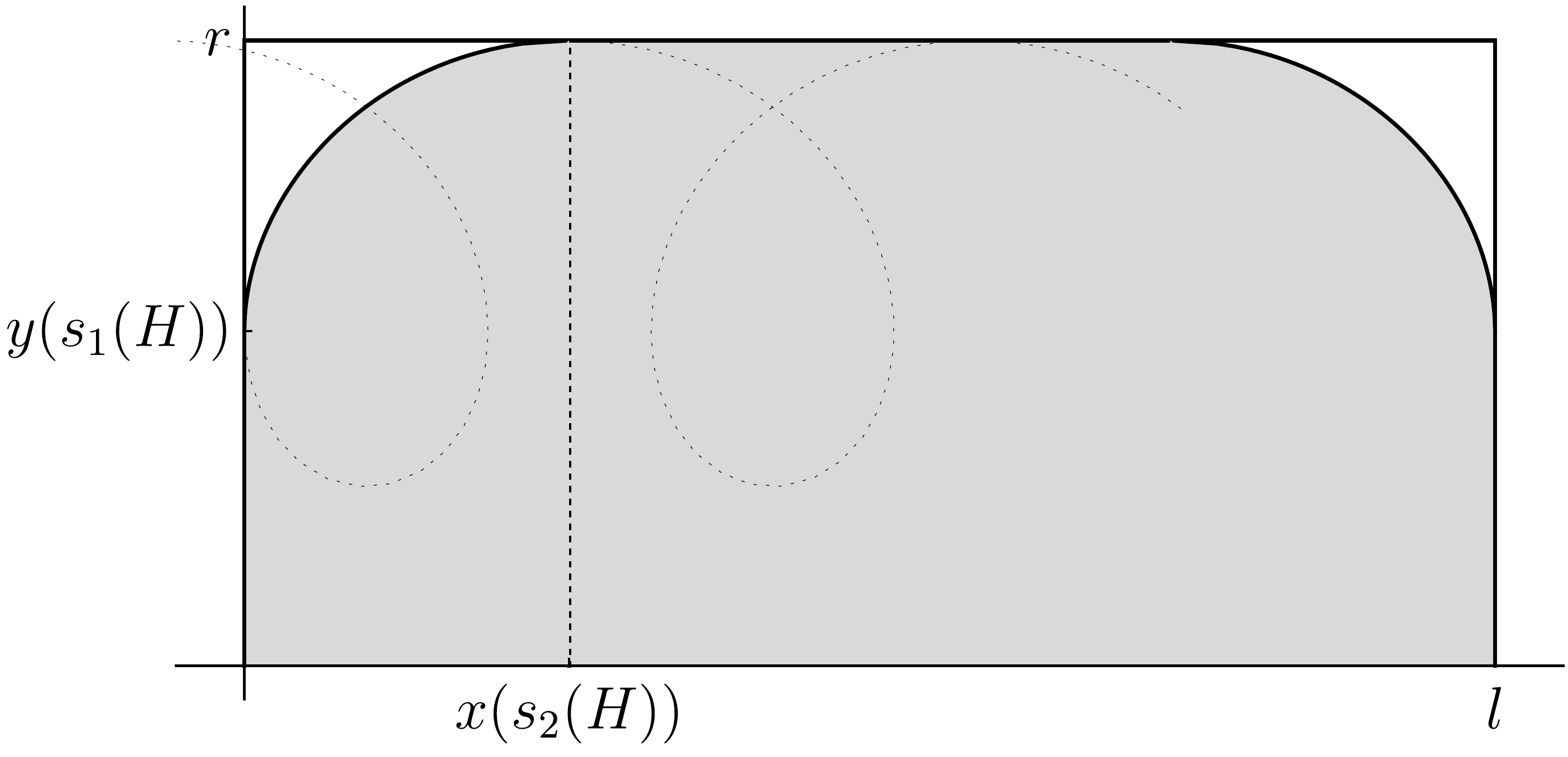}
	\caption{$n=3$. The generating set (gray) of the Cheeger set of $Z_{l,r}$ with $l=3$ and $r=1$.}
	\label{fig:cyl1}
\end{figure}

Therefore, we have shown that the generating curve of $C$ consists of two vertical segments, one horizontal segment, and two portions of a nodoid joining them, see Figure \ref{fig:cyl1}.
Consider now all candidates for the Cheeger set having the same geometric structure as $C$, and notice that each such candidate is uniquely determined by the choice of the mean curvature. The range $\mathcal{I}$ of admissible mean curvatures is dictated by the geometric admissibility of candidates.
Thus, the Cheeger constant $h(Z_{l,r})$ can be characterized through the smallest positive root of the following equation in the variable $H$:
\begin{equation}\label{cyl1:H}
(n-1)H = 
\frac{S_0(H)+S_1(H)+S_2(H)}{V_1(H)+V_2(H)},
\quad H \in \mathcal{I}.
\end{equation}
Equivalently, $h(Z_{l,r})$ can be found as the minimizer of the right-hand side of \eqref{cyl1:H} with respect to $H \in \mathcal{I}$.
Here, for suitable points $s_1(H)$ and $s_2(H)$ (see Figure \ref{fig:cyl1}), $S_0(H)$ denotes the volume of $(n-1)$-dimensional ball of radius $y(s_1(H))$, $S_1(H)$ and $V_1(H)$ stand for the surface area and volume of the portion of the nodoid of mean curvature $H$ generated by the curve $(x(s),y(s))$, $s \in [s_1(H),s_2(H)]$, and, finally, $S_2(H)$ and $V_2(H)$ are the lateral surface area and volume of the cylinder $Z_{l/2-x(s_2(H)),r}$. 
(We define all these quantities on the left half of $Z_{l,r}$ due to the symmetry of candidates with respect to $\{x=\frac{l}{2}\}$.)
More precisely, we have 
\begin{align*}
S_0(H) &= \omega_{n-1} y(s_1(H))^{n-1},\\
S_2(H) &= (n-1) \omega_{n-1} r^{n-2}\left(\frac{l}{2}-x(s_2(H))\right),\\
V_2(H) &= \omega_{n-1} r^{n-1} \left(\frac{l}{2}-x(s_2(H))\right).
\end{align*}
As for $S_1(H)$ and $V_1(H)$, it is convenient to parametrize $(x(s),y(s))$ for $s \in [s_1(H),s_2(H)]$ as $(x,x(y))$, where 
\begin{equation*}\label{eq:cyl0}
x(y) = \int_{y(s_1(H))}^y \left[\left(\frac{t^{n-2}}{T+Ht^{n-1}}\right)^2-1\right]^{-1/2} \, dt,
\end{equation*}
see \eqref{cyl1:hy15} and \eqref{cyl1:hy2}.
Hence, we have
\begin{align*}
S_1(H) &= (n-1)\omega_{n-1} \int_{y(s_1(H))}^r y^{n-2} (1+x'(y)^2)^{1/2} \, dy,\\
V_1(H) &= 
\omega_{n-1} 
x(s_2(H))r^{n-1}
-
(n-1)\omega_{n-1}\int_{y(s_1(H))}^r y^{n-2} x(y) \, dy.
\end{align*}
see, e.g., \cite[(5.3) and (4.3)]{AA}.

Several explicit values of $h(Z_{l,r})$ for different choices of $n$, $l$, and $r$ are listed in Tables \ref{tab:cyl1}, \ref{tab:cyl2}, and \ref{tab:cyl3}.

\begin{table}[!h]
	\centering
	\caption{Values of $H$ and $h(Z_{l,r})$ for $l=1$, $r=1$, and different $n$.}
	\begin{tabular}{| c || c | c | c | c | c | c | c | c| c |}
		\hline
		$n = $  & 3 & 4	& 5	& 10 & 30 \\
		\hline
		\hline
		$H \approx $ 		  & 1.86237 & 1.53976 & 1.38214 & 1.13465 & 1.02474 \\
		\hline
		$h(Z_{l,r}) \approx $ & 3.72474 & 4.61928 & 5.52854 & 10.2118 & 29.7175 \\
		\hline
	\end{tabular}
\label{tab:cyl1}
\end{table}

\begin{table}[!h]
	\centering
	\caption{Values of $H$ and $h(Z_{l,r})$ for $l=2$, $r=1$, and different $n$.}
	\begin{tabular}{| c || c | c | c | c | c | c | c | c| c |}
		\hline
		$n = $  & 3 & 4	& 5	& 10 & 30 \\
		\hline
		\hline
		$H \approx $ 		  & 1.40106 & 1.24549 & 1.17083 & 1.05746 & 1.01027 \\
		\hline
		$h(Z_{l,r}) \approx $ & 2.80212 & 3.73646 & 4.68334 & 9.51714 & 29.2978 \\
		\hline
	\end{tabular}
\label{tab:cyl2}
\end{table}

\begin{table}[!h]
	\centering
	\caption{Values of $H$ and $h(Z_{l,r})$ for $l=3$, $r=1$, and different $n$.}
	\begin{tabular}{| c || c | c | c | c | c | c | c | c| c |}
		\hline
		$n = $  & 3 & 4	& 5	& 10 & 30 \\
		\hline
		\hline
		$H \approx $ 		  & 1.25659 & 1.15544 & 1.10738 & 1.03555 & 1.00634 \\
		\hline
		$h(Z_{l,r}) \approx $ & 2.51318 & 3.46631 & 4.42954 & 9.31991 & 29.184 \\
		\hline
	\end{tabular}
\label{tab:cyl3}
\end{table}

\subsection{Double cones}
Define a double cone $K_{l,r,\theta}$ as a rotationally invariant domain in $\mathbb{R}^n$ whose generating curve $\Gamma$ bounds the triangle with basis $[-l,r] \times \{0\}$, $l,r>0$, the left angle $\theta \in (0,\frac{\pi}{2})$, and the right angle  $\varphi = \arctan(\frac{l}{r}\tan \theta)$.
That is, $\Gamma$ consists of two segments
\begin{equation}\label{eq:cone:lines}
y = (l+x) \tan \theta 
\quad \text{for} \quad x \in [-l,0],
\quad \text{and} \quad 
y = \frac{l}{r} (r-x) \tan \theta 
\quad \text{for} \quad x \in [0,r],
\end{equation}
which intersect at $x=0$, see Figure \ref{fig:double_cone1}.

\begin{figure}[ht]
	\centering
	\includegraphics[height=4.5cm]{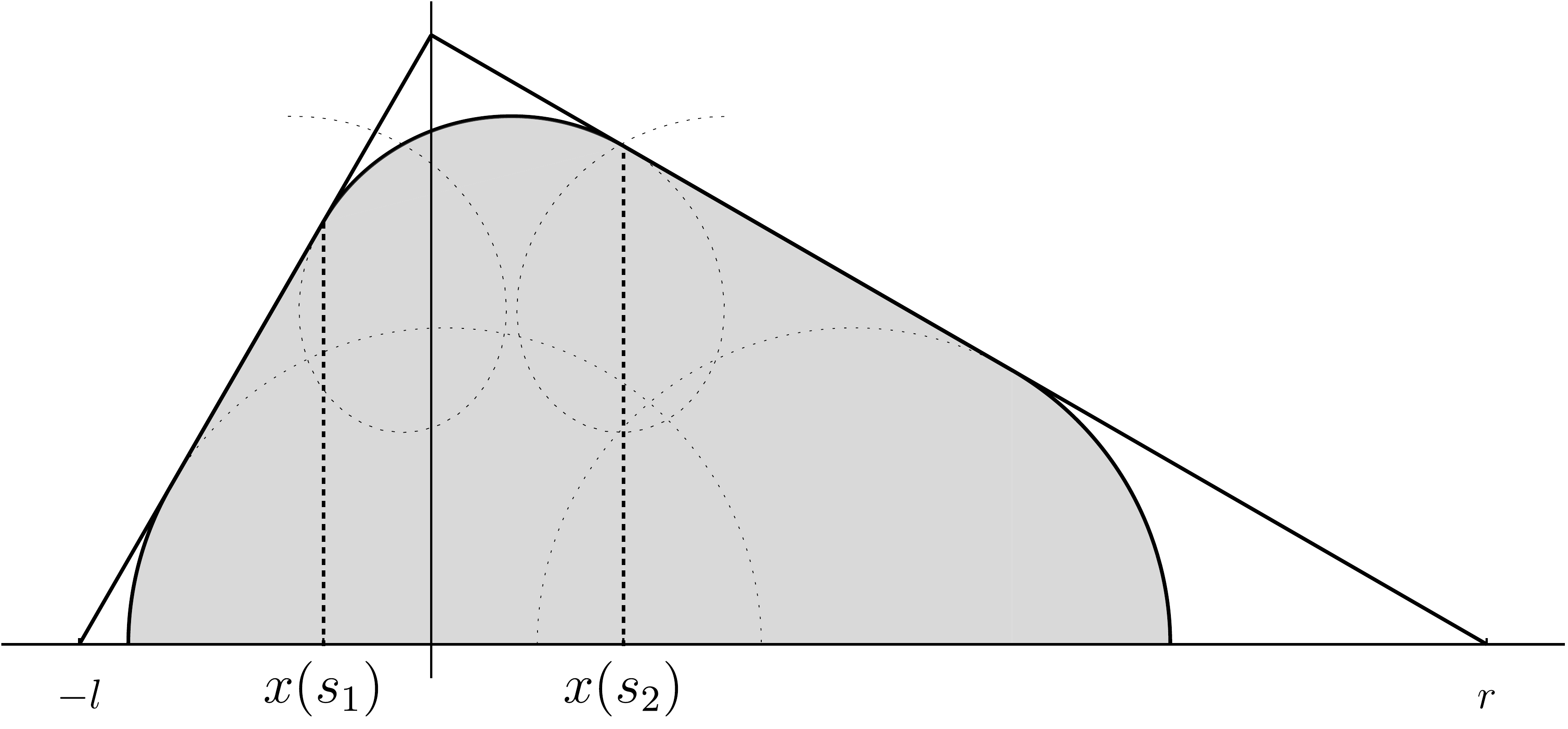}
	\caption{$n=3$. The generating set (gray) of the Cheeger set of $K_{l,r,\alpha}$ with $l=1$, $r=3$, and $\theta=\frac{\pi}{3}$.}
	\label{fig:double_cone1}
\end{figure}

Lemma \ref{propertiesconvex} implies that $K_{l,r,\theta}$ possesses a unique Cheeger set $C$, which is convex, rotationally invariant, and with boundary of class $\mathcal{C}^{1,1}$. 
Moreover, Theorem \ref{thm:main} and the regularity of $C$ yield that the generating curve $\gamma=(x(s),y(s))$, $s \in [a,b]$, of $C$ smooths all three corners of $\Gamma$ either by circular arcs or by parts of nodoids. 
In particular, the free boundary $\partial C \cap K_{l,r,\theta}$ has nonempty interior. 
In view of the properties of Delaunay surfaces (see Section \ref{sec:delaunay}), $\gamma$ can smooth the left and right corners of $\Gamma$ only by circular arcs. Let us show that $\gamma$ smooths the middle corner of $\Gamma$ by a part of a nodoid. Suppose, by contradiction, that $\gamma$ is again a circular arc near the middle corner of $\Gamma$. Then $\gamma$ has to be a half-circle with starting and ending points on the $x$-axis, that is, $C$ is a ball $B_R$ of radius $R=\frac{1}{H}$. However, this is impossible since
$$
\frac{n-1}{R} = (n-1)H = h(K_{l,r,\theta}) = \frac{P(B_R)}{|B_R|} = \frac{n}{R}.
$$

Therefore, we have shown that the generating curve of $C$ consists of two circular arcs connecting the $x$-axis with the lines \eqref{eq:cone:lines}, two segments, each of which belongs to one of the lines \eqref{eq:cone:lines}, and a portion of a nodoid joining the last two segments, see Figure \ref{fig:double_cone1}.
As in Section \ref{sec:cylinder}, consider now all candidates for the Cheeger set having the same geometric structure, and denote by $\mathcal{I}$ the range of mean curvatures for which the candidates are geometrically admissible.
Notice, however, that for any fixed $H \in \mathcal{I}$ the candidate is not necessarily unique. 
(Numerical analysis indicates that there are at most two candidates associated with a fixed $H$.)
Thus, one can characterize the Cheeger constant of $K_{l,r,\alpha}$ as 
\begin{equation}\label{cone:H}
h(K_{l,r,\alpha}) = \inf_{H \in \mathcal{I}} \inf_{T < 0} 
\frac{S_1(H)+S_2(H)+S_3(H,T)+S_4(H,T)+S_5(H,T)}{V_1(H)+V_2(H)+V_3(H,T)+V_4(H,T)+V_5(H,T)}.
\end{equation}
Here $S_1(H)$, $S_2(H)$, $V_1(H)$, and $V_2(H)$ denote the lateral surface areas and volumes of the left and right spherical caps, $S_3(H,T)$, $S_4(H,T)$, $V_3(H,T)$, and $V_4(H,T)$ stand for the lateral surface areas and volumes of the conical frustums generated by the segments of lines \eqref{eq:cone:lines} which join the spherical caps and the portion of the nodoid of mean curvature $H$ and parameter $T$ from \eqref{eq:firstint}, and, finally, $S_5(H,T)$ and $V_5(H,T)$ are the surface area and volume of that portion of the nodoid.

Let us discuss more precisely the quantities in \eqref{cone:H} for a fixed $H$ 	in the three-dimensional case using the parametrization \eqref{eq:kenmotsu} with a displacement $c \in \mathbb{R}$ along the $x$-axis:
\begin{equation}\label{eq:cone:param}
(x(s),y(s)) = \left(\int_0^s \frac{1+B \cos 2Ht}{(1+B^2+2B\cos 2Ht)^{1/2}} \, dt+c,\frac{(1+B^2+2B\cos 2Hs)^{1/2}}{2H}\right).
\end{equation}
In this case, the parameter $B$ is a reparametrization of the parameter $T$.

First, we find $s_1<0$, $s_2>0$, and $B, c \in \mathbb{R}$ such that the nodoid $(x(s),y(s))$ smooths the corner between the lines \eqref{eq:cone:lines} for $s \in [s_1,s_2]$, as the solutions of the following system of four equations:
\begin{align}
\label{eq:ys1}
&y(x(s_1)) = \frac{(1+B^2+2B\cos 2Hs_1)^{1/2}}{2H} = (l+x(s_1))\tan \theta,\\
\label{eq:ys0}
&\frac{dy}{dx}(x(s_1)) = -\frac{B \sin(2Hs_1)}{1+B\cos(2Hs_1)}= \tan \theta,
\end{align}
and 
\begin{align}
\label{eq:ys2} 
&y(x(s_2)) =\frac{(1+B^2+2B\cos 2Hs_2)^{1/2}}{2H} = \frac{l}{r}(r-x(s_2))\tan \theta,\\
\label{eq:ys3}
&\frac{dy}{dx}(x(s_2)) = -\frac{B \sin(2Hs_2)}{1+B\cos(2Hs_2)}= -\frac{l}{r}\tan \theta.
\end{align}
From \eqref{eq:ys0} and \eqref{eq:ys3}, we get
\begin{align*}
s_1 &= -\frac{1}{H}\arctan{\left(\frac{\sqrt{(B^2-1)\tan^2 \theta+B^2}-B}{(B-1) \tan \theta} \right)},\\
s_2 &= \frac{1}{H}\arctan{\left(\frac{\sqrt{l^2 (B^2-1) \tan^2 \theta + B^2r^2}-Br}{l(B-1)\tan \theta} \right)}.
\end{align*}
Then, substituting $s_1$ and $s_2$ into \eqref{eq:ys1} and \eqref{eq:ys2}, respectively, we obtain $B$ and $c$. Recall that the roots $B$ and $c$ are not necessarily unique. 
With the knowledge of the parameters $s_1$, $s_2$, and $B$, we have
\begin{align*}
S_3(H,B) &= \frac{\pi}{H} \int_{s_1}^{s_2} (1+B^2+2B\cos 2Ht)^{1/2} \, dt,\\
V_3(H,B) &= \frac{\pi}{4H^2} \int_{s_1}^{s_2} (1+B \cos 2Ht)(1+B^2+2B\cos 2Ht)^{1/2} \, dt.
\end{align*}

Second, we round the left and right corners by circular arcs of radius $R = \frac{1}{H}$.
It is not hard to see that these arcs are centred at
$$
\left(-l + \frac{R}{\sin \theta}, 0\right)
\quad \text{and} \quad
\left(r - \frac{R}{\sin \varphi}, 0\right).
$$
Moreover, the arcs touch the lines \eqref{eq:cone:lines} at the points
$$
(\hat{x}_1,\hat{y}_1) = \left(-l + \frac{R \cos^2 \theta}{\sin \theta}, R \cos \theta\right)
\quad \text{and} \quad 
(\hat{x}_2,\hat{y}_2) = \left(r - \frac{R \cos^2 \varphi}{\sin \varphi}, R \cos \varphi\right).
$$
Therefore, we get  
\begin{align*}
&S_1(H) = 2 \pi R^2(1-\sin \theta),
\quad
&&S_2(H) = 2 \pi R^2(1-\sin \varphi),\\
&V_1(H)  = \frac{\pi R^3}{3}\left(2-3 \sin \theta + \sin^3 \theta\right),
\quad 
&&V_2(H) = \frac{\pi R^3}{3}\left(2-3 \sin \varphi + \sin^3 \varphi\right).
\end{align*}

Finally, the remaining quantities are given by 
\begin{align*}
S_4(H,B) &= 
\frac{\pi \tan \theta}{\cos \theta} \left((l+x(s_1))^2 - (l+\hat{x}_1)^2\right),\\
S_5(H,B) &=  
\frac{l \pi \tan \theta}{r^2}\sqrt{r^2+l^2 \tan^2 \theta} \left((r-\hat{x}_2)^2-(r-x(s_2))^2\right),
\end{align*}
and 	
\begin{align*}
V_4(H,B) &= 
\frac{\pi \tan^2 \theta}{3} \left((l+x(s_1))^3 - (l+\hat{x}_1)^3\right),\\
V_5(H,B) &= 
\frac{\pi l^2 \tan^2 \theta}{3r^2} \left((r-\hat{x}_2)^3-(r-x(s_2))^3\right).
\end{align*}

Several explicit values of $h(K_{l,r,\theta})$ for $n=3$ and different choices of $l$, $r$, and $\theta$ are listed in Table \ref{tab:doublecone}.

\begin{table}[!h]
	\centering
	\caption{Values of $h(K_{l,r,\theta})$ for $n=3$ and different parameters.}
	\begin{tabular}{| c || c | c | c | c | c | c | c | c | c | c |}
		\hline
		$l = $  & 9/5	  & 1 	& 1	& 1		  & 1	&1	 \\
		\hline
		$r = $ 	& 16/5	  & 3 	& 1	& 1 	  & 1 	&1	 \\
		\hline
		$\theta = $ & $\arcsin(4/5)$  & $\pi/3$ & $2\pi/5$ & $\pi/3$ & $\pi/4$ & $\pi/6$  \\
		\hline
		\hline
		$h(K_{l,r,\theta}) \approx $ & 1.6502 & 2.22333 & 2.38303 & 3.00582 & 4.00593 & 5.75003  \\
		\hline
	\end{tabular}
\label{tab:doublecone}
\end{table}

\subsection{Cones}\label{sec:cone}
We define the cone $K_{l,\theta}$ as the rotationally invariant domain in $\mathbb{R}^n$ whose generating curve $\Gamma$ bounds the right triangle with leg $[-l,0] \times \{0\}$, $l>0$, and the left angle $\theta \in (0,\frac{\pi}{2})$. 
That is, $\Gamma$ consists of two segments
\begin{equation*}\label{eq:simplecone:lines}
y = (l+x) \tan \theta 
\quad \text{for} \quad x \in [-l,0],
\quad \text{and} \quad 
x = 0 
\quad \text{for} \quad y \in [0, l \tan \theta],
\end{equation*}
see Figure \ref{fig:cone1}. 
Observe that the cone $K_{l,\theta}$ can be obtained as the limit case of the double cone $K_{l,r,\theta}$ for $r=0$. Again by Lemma  \ref{propertiesconvex}, the unique Cheeger set $C$ of $K_{l,\theta}$ is convex, rotationally invariant, and with boundary of class $\mathcal{C}^{1,1}$. Arguing as in the previous subsection, we deduce that $\gamma$, the generating curve of $C$, smooths the left corner of $\Gamma$ by a circular arc, and the top corner of $\Gamma$ by a portion of a nodoid. 
It is then possible to characterize $h(K_{l,\theta})$ as the infimum of ratio of the perimeter over volume of the candidates for the Cheeger set as in \eqref{cone:H}. 
Several explicit values of $h(K_{l,\theta})$ for $n=3$ and different choices of $l$ and $\theta$ are listed in Table \ref{tab:cone}.

\begin{figure}[ht]
	\centering
	\includegraphics[height=4.5cm]{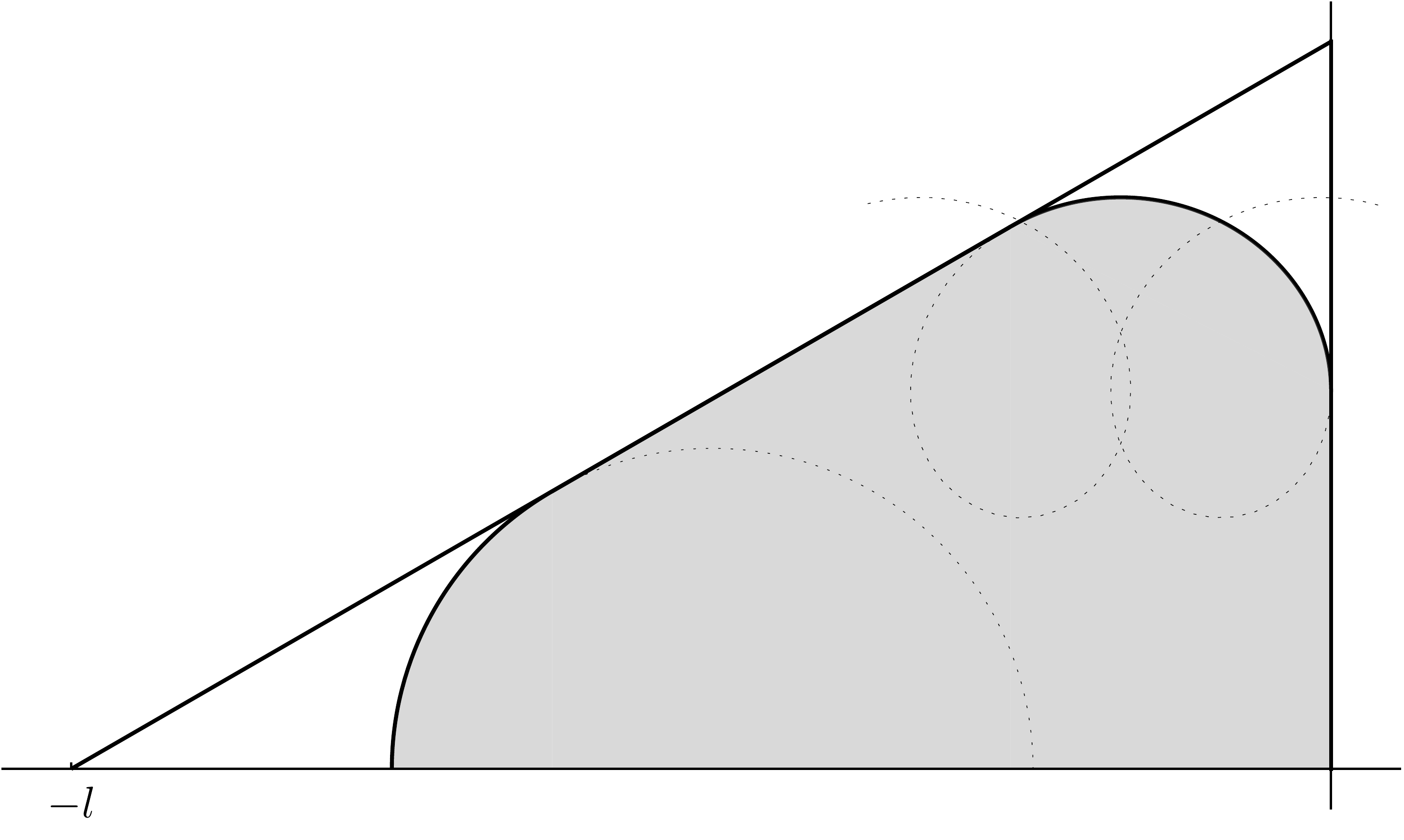}
	\caption{$n=3$. The generating set (gray) of the Cheeger set of $K_{l,\alpha}$ with $l=1$ and $\theta=\frac{\pi}{6}$.}
	\label{fig:cone1}
\end{figure}

\begin{table}[!h]
	\centering
	\caption{Values of $h(K_{l,\theta})$ for $n=3$ and different parameters.}
	\begin{tabular}{| c || c | c | c | c | c | c | c | c | c |}
		\hline
		$l = $  & 4 	& 3 & 1	  & 1 	& 1			  	 \\
		\hline
		$\theta = $ & $\arcsin(3/5)$ & $\arcsin(4/5)$ & $\pi/3$ & $\pi/4$ & $\pi/6$        \\
		\hline
		\hline
		$h(K_{l,r,\theta}) \approx $ & 1.69452 & 1.71916 & 4.6575  & 5.86018 & 7.85898      \\
		\hline
	\end{tabular}
\label{tab:cone}
\end{table}

\section{On the presence of unduloids and cylinders}
In Theorem \ref{thm:main} and Propositions \ref{prop:unlul-nonexist1} and \ref{lem:unlul-nonexist}, we have shown that, under several assumptions, the free boundary of a Cheeger set of a rotationally invariant domain cannot have a portion of an unduloid or a cylinder as its part. 
It is therefore natural to ask whether there exists a rotationally invariant domain such that a portion of an unduloid or a cylinder can appear in the free boundary of its rotationally invariant Cheeger set. 
In this section, we provide numerical evidence of the existence of such domain.

Consider $\Omega \subset \mathbb{R}^3$ generated by a curve $\Gamma$ which is symmetric with respect to the line $\{x=0\}$ and which is defined in $\mathbb{R}^+ \times \mathbb{R}^+$ by the union of three segments 
\begin{align}
\label{eq:undul:segment1}
y&=-\frac{B-D}{C} x + B
\quad \text{for} \quad x \in [0,C],\\
\label{eq:undul:segment2}
y&=\frac{B-D}{A-C} (x-C) + D
\quad \text{for} \quad x \in [C,A],\\
\notag
x&=A
\quad \text{for} \quad y \in [0,B],
\end{align}
where $A, B, C, D >0$ are some constants such that $A>C$ and $B>D$, see Figure \ref{fig:undul}. 
By construction, $\Omega$ is Schwarz symmetric with respect to the $x$-axis, and hence $\Omega$ possesses a rotationally invariant Cheeger set $C$, see Section \ref{sec:cheeger}. 
Noting, moreover, that the union of Cheeger sets is again a Cheeger set (see, e.g., \cite[Proposition 3.5 (vi)]{leonardi}), and taking the union of $C$ with its reflection with respect to the plane $\{x=0\}$, we can assume that $C$ is symmetric with respect to $\{x=0\}$.

\begin{figure}[ht]
	\centering
	\includegraphics[width=0.7\linewidth]{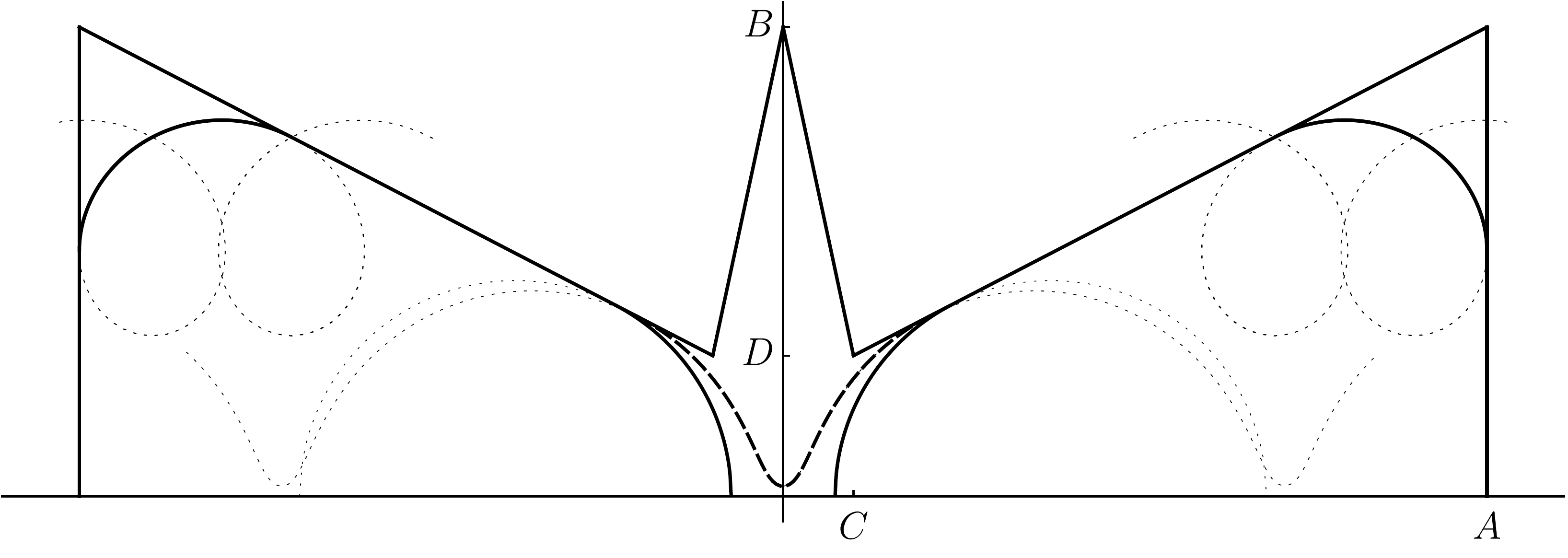}
	\caption{The generating curve of $\Omega$ with $A=3$, $B=2$, $C=0.3$, $D=0.6$, and the generating curves of two nonoptimal candidates for the Cheeger set prolonged by dotted lines. 
	The candidate $C_1$ with circular arcs (case \ref{case:undul4}) has $\frac{P(C_1)}{|C_1|} \approx 2.1742$, and the candidate $C_2$ with the dashed unduloid (case \ref{case:undul3}) has $\frac{P(C_2)}{|C_2|} \approx 2.17616$.}
	\label{fig:undul}
\end{figure}

Since the generating curve $\Gamma$ of $\Omega$ has three convex corners at the points $(A,B)$, $(0,B)$, and $(-A,B)$, 
the generating curve $\gamma$ of $C$ smooths them. This fact can be obtained arguing by contradiction: first, straightforward calculations show that the ratio $\frac{P(C)}{|C|}$ diminishes after truncation of a convex corner by a horizontal segment which is sufficiently close to the top of the corner, and then Proposition \ref{prop:reg} \ref{prop:reg:2} implies smoothness of $\gamma$.
Noting that among admissible Delaunay surfaces only nodoids and spheres have vertical tangents, we see that the corners of $\Gamma$ at the points $(A,B)$ and $(-A,B)$ are smoothed either by portions of a nodoid or by circular arcs. 
As for the behaviour of $\gamma$ near the corner of $\Gamma$ at the point $(0,B)$, there are several possibilities:
\begin{enumerate}[label={\rm(\roman*)}]
	\item\label{case:undul1} there is a portion of a Delaunay surface, different from the cylinder, which is inscribed in the convex corner at the point $(0,B)$ ($\gamma$ is tangent to the segment \eqref{eq:undul:segment1} and its reflection with respect to $\{x=0\}$);
	\item\label{case:undul2} there is a portion of a Delaunay surface which passes through the points $(C,D)$ and $(-C,D)$;
	\item\label{case:undul3} there is a portion of an unduloid which connects in the $\mathcal{C}^1$-fashion the segment \eqref{eq:undul:segment2} with its reflection with respect to $\{x=0\}$;
	\item\label{case:undul4} there is a circular arc which connects the segment \eqref{eq:undul:segment2} with the $x$-axis. In this case, $C$ consists of two connected components. 
\end{enumerate}

Following the methodology of Section \ref{sec:examples}, the Cheeger constant $h(\Omega)$ can be found by minimizing the ratio of the perimeter over volume of candidates for the Cheeger set which are defined by cases \ref{case:undul1}-\ref{case:undul4}. 
We performed corresponding numerical computations with $A=3$, $B=2$, $C=0.3$, and varying $D$.
The results suggest that there exist critical values $D_1 \approx 0.42312$, $D_2 \approx 0.44163$, $D_3 \approx 1.1216$, and $D_4 \approx 1.9282$ such that the generating curve $\gamma$ of $C$ behaves as follows:
\begin{enumerate}[label={\rm(\Roman*)}]
	\item for $D \in (0,D_1]$, case \ref{case:undul4} occurs;
	\item for $D \in [D_1,D_2)$, case \ref{case:undul2} occurs, where the surface is an unduloid having a point of minimum at $x=0$;
	\item for $D = D_2$, case \ref{case:undul2} occurs, where the surface is a cylinder;
	\item\label{case:I} for $D \in (D_2,D_3)$, case \ref{case:undul2} occurs, where the surface is an unduloid having a point of maximum at $x=0$;
	\item for $D = D_3$, case \ref{case:undul2} occurs, where the surface is a sphere;
	\item for $D \in (D_3,D_4]$, case \ref{case:undul2} occurs, where the surface is a nodoid;
	\item for $D \in [D_4, B)$, case \ref{case:undul1} occurs, where the surface is a nodoid.
\end{enumerate}
In particular, case \ref{case:undul3} was not observed.
On Figures \ref{fig:undul2} and \ref{fig:undul}, we depict the Cheeger set and nonoptimal candidates for the Cheeger set, respectively, by choosing $D=0.6$ which corresponds to case \ref{case:I}.

 \begin{figure}[ht]
	\centering
	\includegraphics[width=0.7\linewidth]{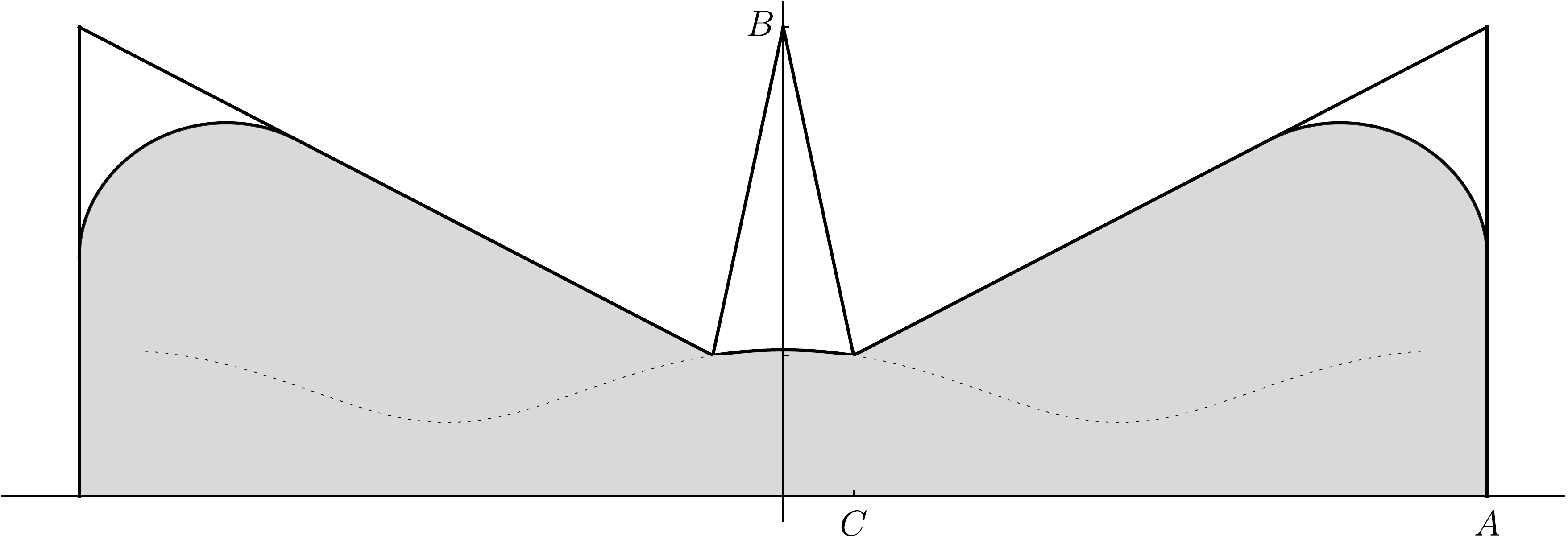}
	\caption{The generating set (gray) of the Cheeger set of $\Omega$ with $A=3$, $B=2$, $C=0.3$, $D=0.6$, and $h(\Omega) \approx 2.13324$.}
	\label{fig:undul2}
\end{figure}

\section{Comments and remarks}

\begin{rem}
The problem of finding a rotationally invariant Cheeger set amounts to determine the solution of a weighted Cheeger problem in a set $D \subset \R \times \R^+$, where the weighted perimeter and volume are given, for every $E \subset \R \times \R^+$, by
\[ 
P_w(E) := \int_{\partial^* E} y^{n-2}\,d\mathcal{H}^{1} 
\quad \text{and} \quad
V_w(E) := \int_{E} y^{n-2}\,dx\,dy.
\]
This kind of problem was introduced in \cite{ionesculachand} for a general class of weights which, however, does not include the case under consideration. The isoperimetric problem in $\R \times \R^+$ with the weight $y^\alpha$, $\alpha > 0$, for both perimeter and volume has been first studied in \cite{madernasalsa}.   
\end{rem}

\begin{rem}
In Proposition \ref{prop:unlul-nonexist1} we assumed that $\Omega$ admits a rotationally invariant Cheeger set whose generating curve $\gamma :[a,b] \to \R \times (0,+\infty)$ is closed, convex, and of class $\mathcal{C}^{1,1}$.
We anticipate that the existence and, moreover, uniqueness of such Cheeger set holds true provided $\Omega$ is generated by a closed, convex curve $\Gamma :[\alpha,\beta] \to \R \times (0,+\infty)$. 
In some particular cases, such as the torus in $\R^n$, this has been proven in \cite{KLV}. 
\end{rem}

\begin{rem}
	Recall that the Cheeger set of a planar, convex, bounded domain $\Omega$ can be characterized by ``rolling'' a disk $B_r(x)$ of radius $r=\frac{1}{h(\Omega)}$ inside $\Omega$, see \eqref{eq:cunion}.
	In particular, there exists at least one $x \in \Omega$ such that $B_r(x) \subset \Omega$. 
	It is then natural to wonder whether a similar characterization of the Cheeger set of a convex, rotationally invariant, bounded domain $\Omega \subset \mathbb{R}^n$ can be given, where besides balls of radius $\frac{n-1}{h(\Omega)}$ one can also use appropriately defined nodoidal ``caps'' due to Theorem \ref{thm:main}.
	It can happen, however, that no ball of radius $\frac{n-1}{h(\Omega)}$ is contained in $\Omega$, even if a spherical cap of the same radius is a connected component of the free boundary of the corresponding Cheeger set. 
	Indeed, consider a three-dimensional cone $K_{l,\theta}$ with some $l>0$ and $\theta \in (0,\frac{\pi}{2})$ defined as in Section \ref{sec:cone}. 
	Clearly, we have
	$$
	h(K_{l,\theta}) < \frac{P(K_{l,\theta})}{|K_{l,\theta}|} = \frac{\pi l^2 + \frac{\pi l^2}{\cos \theta}}{\frac{1}{3} \pi l^3 \tan \theta} = \frac{3(1+\cos \theta)}{l \sin \theta}.
	$$ 
	On the other hand, the radius of the maximal ball inscribed in $K_{l,\theta}$ equals $\frac{l \sin \theta}{1+\sin\theta}$.
	Thus, no ball of radius $\frac{2}{h(\Omega)}$ can be inscribed in $K_{l,\theta}$ provided $\frac{2 l \sin \theta}{3(1+\cos \theta)} > \frac{l \sin \theta}{1+\sin\theta}$. It is not hard to see that this inequality is satisfied for all $\theta$ sufficiently close to $\frac{\pi}{2}$, which establishes the counterexample.
		
	In the same spirit, it is also natural to ask whether a general characterization of the Cheeger set and constant of a convex domain of revolution generated by a polygonal curve can be provided, as was done in \cite[Theorem 3]{kawohllachand} for planar polygons. 
	However, such a characterization does not seem straightforward to obtain.
\end{rem}

\bigskip
\noindent
{\bf Acknowledgments.}
The essential part of the present research was performed during a visit of E.P. at the University of West Bohemia and a visit of V.B. at Aix-Marseille University. 
The authors wish to thank the hosting institutions for the invitation and the kind hospitality. 
V.B. was supported by the project LO1506 of the Czech Ministry of Education, Youth and Sports, and by the grant 18-03253S of the Grant Agency of the Czech Republic.
The authors also wish to thank the anonymous reviewer for valuable comments and suggestions.

\end{document}